\documentclass[A4,12pt,reqno]{amsart}
\usepackage{booktabs}
\usepackage{amsthm}
\usepackage{graphicx}
\usepackage{enumitem}

\theoremstyle{plain}
\newtheorem{theorem}{Theorem}[section]

\theoremstyle{definition}

\newtheorem{example}{Example}[section]

\theoremstyle{remark}
\newtheorem{remark}{Remark}

\allowdisplaybreaks 

\begin{document}

\title[b-Splines]{b-Spline Curves and Surfaces as a Minimization of Quadratic Operators}

\author{Svetoslav I. Nenov}
\address{Department of Mathematics, University of Chemical Technology and Metallurgy, Sofia 1756, Bulgaria}
\email{nenov@uctm.edu}

\begin{abstract}
The goal of this short note is to prove that every b-spline curve or surface (generated by uniform knots, without multiplicity) may be defined as minimum of positive quadratic operator.
\end{abstract}

\maketitle

\section{Introduction}\label{sec:introduction}

The b-spline curves and surfaces are an essential tool in many engineering software for design and visualization -- for example ANSYS, RFEM 3D, etc. So, it is necessary to have in these applications many different methods to construct b-spline elements. Also see \cite{Dechevsky1}, \cite{Dechevsky2}, and references therein.

We will prove that any b-spline curve or surface minimizes positive quadratic operator: appropriate moving least-square error. 

Let us mark that different approaches in moving least-squares method are used by Shepard -- computer software SYMAP (Harvard Laboratory for Computer Graphics), Lancaster in 1979, and the works of D. Levin in 1999, see \cite{Levin3}. In \cite{stratiev} it has been shown that moving least-squares method is an adequate mathematical tool for determining diesel-fuel cetane-number (or cetane-index) from easily available physical properties of fuels.

In this section we will remind the definition of b-splines generated by control points in ${\mathbb R}^{d+1}$ and definition of moving least-squares approximation for a given data set $\{(\boldsymbol x_i, f(\boldsymbol x_i)):\boldsymbol x_i\in{\mathbb R}^{d}\}\subset {\mathbb R}^{d+1}$.

\section{Preliminaries}\label{sec:1}

\subsection{b-Splines}

Let $\{\boldsymbol p_i\in{\mathbb R}^{d+1}: i=0,\dots,n\}$ be a set of $n+1$ (control) points.

Let $r$ be an integer, $1\leq r \leq n+1$ (the order of spline).

We will use uniform knots, without multiplicity: $t_i=i$, $i=0,\dots,n+r$. 

Using  Cox-de Boor recursion formula (see \cite{lee1}, \cite{lee2}), let us define the following basis functions:
\begin{gather}\label{b-spline_eq_03}
B_{i,1}(t)=\begin{cases}
1,&\quad \text{if}\ t_i\leq t< t_{i+1},\\
0,&\quad \text{otherwise},
\end{cases}
\end{gather}
for $0\leq i\leq n+r-1$; and
\begin{gather}\label{b-spline_eq_04}
\begin{split}
B_{i,j}(t)=&\dfrac{t-t_i}{t_{i+j-1}-t_i}B_{i,j-1}(t)+\dfrac{t_{i+j}-t}{t_{i+j}-t_{i+1}}B_{i+1,j-1}(t)\\
=&\dfrac{t-i}{j-1}B_{i,j-1}(t)+\dfrac{i+j-t}{j-1}B_{i+1,j-1}(t),
\end{split}
\end{gather}
for $2\leq j \leq r$, $0\leq i\leq n+r-j$.

The b-spline curve of order $r$ is defined as a linear combination of control points $\boldsymbol p_i$:
\begin{gather}\label{b-spline_eq_05}
\boldsymbol \gamma (t) = \sum\limits_{i=0}^{n} B_{i,r}(t)\boldsymbol p_i,\quad t\in[t_{r-1},t_{n+1}].
\end{gather}

%We will deal with cubic b-splines, i.e. $d=4$.

\subsection{Moving Least-Squares Method}
Let:
\begin{enumerate}
\item $\mathcal D$ be a bounded domain in $\mathbb R^d$.
\item $\boldsymbol x_i\in \mathcal D$, $i=0,\dots,m$; $\boldsymbol x_i\not=\boldsymbol x_j$, if $i\not=j$.
\item $f: \mathcal D\to \mathbb R$ be a continuous function.
\item $p_i: \mathcal D\to \mathbb R$ be continuous functions, $i=1,\dots,l$. The functions $\{p_1,\dots,p_l\}$ are linearly independent in $\mathcal D$ and let $\mathcal P_l$ be their linear span.
\item $W:(0,\infty)\to(0,\infty)$ is a strictly positive functions.
\end{enumerate}

Usually the basis in $\mathcal P_l$ is constructed by monomials. For example: $p_l(\boldsymbol x)=x_1^{k_1}\dots x_d^{k_d}$, where $\boldsymbol x=(x_1,\dots,x_d)$, $k_1,\dots k_d\in\mathbb N$, $k_1+\dots+k_d\leq l-1$. In the case $d=1$, the standard basis is $\{1,x,\dots,x^{l-1}\}$.

Following \cite{AlexaBehrCohen-Or}, \cite{Levin1}, \cite{Levin2}, \cite{Levin3}, we use the following definition. The {\it moving least-squares approximation} of order $l$ at a fixed point $\boldsymbol x$ is the value of $p^*(\boldsymbol x)$, where $p^*\in\mathcal P_l$ is minimizing the least-squares error
\begin{gather}\label{MLSM-1}
\sum_{i=1}^{m}W(\|\boldsymbol x-\boldsymbol x_i\|)\left(p(\boldsymbol x)-f(\boldsymbol x_i)\right)^2
\end{gather}
among all $p\in\mathcal P_l$.

The approximation is ``local'' if weight function $W(s)$ is fast decreasing as $s$ tends to infinity. Interpolation is achieved if $W(0)=\infty$. We define additional function $w:[0,\infty)\to[0,\infty)$, such taht:
$$w(s)=
\begin{cases}
\dfrac1{W(s)},&\quad \text{if ($s>0$) or ($s=0$ and $W(0)<\infty$)},\\
0,&\quad \text{if ($s=0$ and $W(0)=\infty$)}.
\end{cases}
$$

Some examples of $W(s)$ and $w(s)$, $s\geq 0$:
\begin{alignat*}{2}
&W(s)=e^{-\alpha^2s^2}&&\qquad \text{$\exp$-weight},\\
&W(s)=s^{-\alpha^2}&&\qquad \text{Shepard weights},\\
&w(s)=s^2e^{-\alpha^2s^2}&&\qquad \text{McLain weight},\\
&w(s)=e^{\alpha^2s^2}-1&&\qquad \text{see Levin's works}.
\end{alignat*}

Here and below: the superscript $^t$ denotes transpose of matrix; $I$ is the identity matrix.

Let us introduce the matrices:
\begin{align*}
E=& \begin{pmatrix}
p_1(\boldsymbol x_1) & p_2(\boldsymbol x_1) & \cdots & p_l(\boldsymbol x_1)\\
p_1(\boldsymbol x_2) & p_2(\boldsymbol x_2) & \cdots & p_l(\boldsymbol x_2)\\
\vdots   & \vdots   &        & \vdots\\
p_1(\boldsymbol x_m) & p_2(\boldsymbol x_m) & \cdots & p_l(\boldsymbol x_m)\\
\end{pmatrix},\ \boldsymbol a = \begin{pmatrix}
a_1\\
a_2\\
\vdots\\
a_m
\end{pmatrix},\ \boldsymbol c = \begin{pmatrix}
p_1(\boldsymbol x)\\
p_2(\boldsymbol x)\\
\vdots\\
p_l(\boldsymbol x)
\end{pmatrix}\\
D= & 2 \begin{pmatrix}
w(\|\boldsymbol x-\boldsymbol x_1\|) & 0 & \cdots & 0\\
0 & w(\|\boldsymbol x-\boldsymbol x_2\|) & \cdots & 0\\
\vdots   & \vdots   &        & \vdots\\
0 & 0 & \cdots & w(\|\boldsymbol x-\boldsymbol x_m\|)\\
\end{pmatrix}.
\end{align*}

Through the article, we assume the following conditions (H1):
\begin{enumerate}
\item[(H1.1)] $1\in \mathcal P_l$.
\item[(H1.2)] $1\leq l \leq m$.
\item[(H1.3)] rank$(E^t)=l$.%, $1\in\mathcal P_l$.
\item[(H1.4)] $w$ is a smooth function.
\end{enumerate}

%The equivalent statement of the moving least-squares minimization problem is  the following constrained problem:
%\begin{align}
%&\text{Find the minimum of}\ Q = \sum_{i=1}^{m}w(\boldsymbol x,\boldsymbol x_i)a_i^2,\label{eq:1.1}\\
%&\text{subject to}\ \sum_{i=1}^{m}a_ip_j(\boldsymbol x_i)=p_j(\boldsymbol x),\ j=1,\dots l.\label{eq:1.2}
%\end{align}

\begin{theorem}[see \cite{Levin1}]\label{thm_Levin1} Let the conditions (H1) hold true.

Then:
\begin{enumerate}
\item
The matrix $E^tD^{-1}E$
%\begin{gather}
%A=\begin{pmatrix}
%D & E\\
%E^t & 0
%\end{pmatrix}
%\end{gather}
is non-singular.
\item The approximation defined by the moving least-squares method is
\begin{gather}\label{levin_thm_1}
\hat L (f) = \sum_{i=1}^m a_i f(\boldsymbol x_i),
\end{gather}
where
\begin{gather}\label{levin_thm_2}
\boldsymbol a = A_0 \boldsymbol c\quad\text{and}\quad   A_0= D^{-1}E\left(E^tD^{-1}E\right)^{-1}.
\end{gather}
\item If $w(0)=0$, then the approximation is interpolatory.
\end{enumerate}
\end{theorem}

\section{b-Spline Curve as a Minimum of\\ Moving Least-Squares Error}

Using the definitions and notations introduced in Section \ref{sec:1}, our goal is to prove the following theorem.
\begin{theorem}\label{sec_2_thm_1} Let:
\begin{enumerate}
\item $d=1$, $n,r\in{\mathbb Z_+}$, $r\leq n+1$, $f:[0,n+r]\to\mathbb R$ be a continuous function.
\item $\boldsymbol p_i=(i,f(i))$, $i=0,\dots,n+r$.
\item Let $\boldsymbol \gamma (t)=\left(\begin{smallmatrix}
\gamma_1(t)\\\gamma_2(t)
\end{smallmatrix}\right)$ be the b-spline of order $r$ and knot vector $\{t_i=i: i=0,\dots,n+r\}$.
\end{enumerate}

Then $\gamma_1(t)=t-2$, $t\in[0,n+r]$ and there exists a weight function $W$, such that
$$
\gamma_2 (x) = \hat L(f)(x),\quad x\in[r-1,n+1],
$$
where $\hat L(f)(x)$ is the approximation defined by the moving least-squares method for the data $\{\boldsymbol p_i:i=0,\dots,n+r\}$.
\end{theorem}

\begin{proof} We will prove the theorem for the cubic splines, i.e. $r=4$. The proof for the different orders is similar.

From conditions (2) and (3), if we set $x_i=i$, then $t_i=x_i$, $i=0,\dots,n$ and hence $\gamma_1(t)=t$.

%The basic idea is to compare the corresponding coefficients in linear combinations \eqref{b-spline_eq_05} and \eqref{levin_thm_1}.

The b-spline curve of order $4$, defined using knots $\{t_i=i: i=0,\dots,n+r\}$, is
$$
\gamma_2(t) = \sum\limits_{i=0}^{n} B_{i,4}(t)f(x_i),\quad x\in[t_{r-1},t_{n+1}]\equiv[3,n+1].
$$

\begin{table}[t!]
{\scriptsize\begin{tabular}{llllll}
$\vdots$	       & $\vdots$	    &			  &		      &\\
$t_{i_0-2}<t_{i_0-1}$: & $B_{i_0-2,1}(t)=0$ & \vdots				      &\\
		       &		    & $B_{i_0-2,2}(t)=0$  & \vdots	      &\\
$t_{i_0-1}<t_{i_0}$:   & $B_{i_0-1,1}(t)=0$ & 		          & $B_{i_0-2,3}(t)>0$ & \vdots\\
		       & 		    & $B_{i_0-1,2}(t)>0$  &		      & $B_{i_0-2,4}(t)>0$\phantom{$\vdots$}\\
$t_{i_0}<t<t_{i_0+1}$: & $B_{i_0,1}(t)=1$   & 		       	  & $B_{i_0-1,3}(t)>0$&	\phantom{$\vdots$}		  \\
		       & 		    & $B_{i_0,2}(t)>0$    &		      & $B_{i_0-1,4}(t)>0$\phantom{$\vdots$}\\
$t_{i_0+1}<t_{i_0+2}$: & $B_{i_0+1,1}(t)=0$ & 		       	  & $B_{i_0,3}(t)>0$  & \vdots\\
		       & 		    & $B_{i_0+1,2}(t)=0$  & \vdots	      &\\
$t_{i_0+2}<t_{i_0+3}$: & $B_{i_0+2,1}(t)=0$ & \vdots		  & 		      &\\
$\vdots$	       & $\vdots$	    & 			  &		      &\\
\end{tabular}}\bigskip 

\caption{Cox-de Boor recursion algorithm}\label{Cox_de_Boor}
\end{table}

The following properties of b-spline basis functions $B_{i,j}(t)$ are well known:
\begin{enumerate}
\item[(BS-0)] By the direc calculation (see formulas \eqref{b-spline_eq_03}, \eqref{b-spline_eq_04}, and the schema illustrated in Table \ref{Cox_de_Boor}, $r=4$):
\begin{gather*}
B_{i,4}(t)=\begin{cases}
0,&\ \text{if}\ t<i,\\[4pt]
\dfrac{1}{6}\left( t-i \right) ^{3},&\ \text{if}\ i\leq t<i+1,\\[4pt]
-\dfrac{2}{3}\left( t-i-1 \right) ^{3}+\dfrac{1}{6}\left( t-i \right) ^{3},&\ \text{if}\ i+1\leq t<i+2,\\[4pt] 
\left( t-i-2 \right) ^{3}-\dfrac{2}{3}\left( t-i-1 \right) ^{3}&\\[4pt]
+\dfrac{1}{6}\left( t-i \right) ^{3},&\ \text{if}\ i+2\leq t<i+3,\\[4pt]
-\dfrac{2}{3}\left( t-i-3 \right) ^{3}+ \left( t-i-2 \right) ^{3}&\\[4pt]
-\dfrac{2}{3}\left( t-i-1 \right) ^{3}+\dfrac{1}{6}\left( t-i \right) ^{3},&\ \text{if}\ i+3\leq t<i+4,\\[4pt]
0,&\ \text{if}\ i+4\leq t.
\end{cases}
\end{gather*}
\item[(BS-1)] If $t\in(i,i+j)$, then $B_{i,j}(t)>0$.
\item[(BS-2)] If $t\in[0,i]\cup[i+j,n+j]$, then $B_{i,j}(t)=0$.
\item[(BS-3)] $\sum\limits_{i=0}^{n}B_{i,j}(t)=1$, for any $t\in(j-1,n+1)$.
\item[(BS-4)] $B_{i,j}(t)$ has $C^{j-2}$ continuity at each knot.
\item[(BS-5)] By the simple substitutions in the formulas in (BS-0): 
\begin{alignat*}{2}
B_{i,j}(t+i)=&B_{k,j}(t+k), && t\in(0,j),\\
B_{i-2,j}(t)=&B_{i,j}(t+2), \qquad && t\in(i-2,i+j-2). 
\end{alignat*}
\end{enumerate}

Let $t$ be a fixed point in the interval $(3,n+1)$ and $i_0$ be an integer such taht
$3\leq i_0<t<i_0+1\leq n+1$. Then
\begin{gather}\label{thm1_temp_7}
\gamma_2 (t) = \sum\limits_{i=0}^{n} B_{i,4}(t)f(i)= \sum\limits_{i=i_0-3}^{i_0} B_{i,4}(t)f(i).
%=& \sum\limits_{i=i_0-1}^{i_0+2} B_{i-2,4}(x)f(i-2)\\
%=& \sum\limits_{i=i_0-1}^{i_0+2} B_{i,4}(x+2)f(i-2),
\end{gather}
because if $i=1,\dots,i_0-4,i_0+1,\dots,n$, then $B_{i,4}(x)=0$, see (BS-1) and (BS-2). 

On the other hand, let us consider the moving least-squares problem for the given data $\{(i,f(i)):i=0,\dots,n+4\}$.  Let us set $l=1$, and 
\begin{gather*}
W(x)=\begin{cases}
0,&\ \text{if}\ x<-2,\\[4pt]
\dfrac{1}{6}(x+2)^3,&\ \text{if}\ -2\leq x<-1,\\[4pt]
-\dfrac{1}{2}x^3-x^2+\dfrac{2}{3},&\ \text{if}\ -1\leq x<0,\\[4pt] 
\dfrac{1}{2}x^3-x^2+\dfrac{2}{3},&\ \text{if}\ 0\leq x<1,\\[4pt]
-\dfrac{1}{6}(x-2)^3,&\ \text{if}\ 1\leq x<2,\\[4pt]
0,&\ \text{if}\ 2\leq x,
\end{cases}
\end{gather*}
see Figure \ref{fig:weight}.

\begin{figure}[t!]
\centerline{\includegraphics[scale=.3]{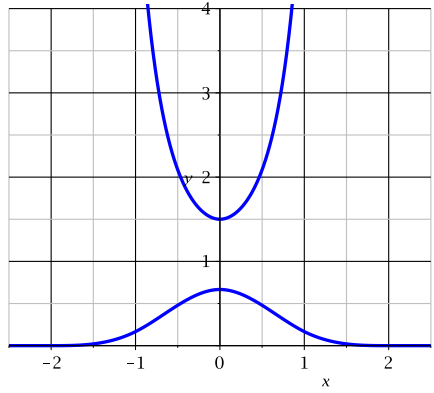}}
\caption{The graphics of $W(x),\ x\in(-2.5,2.5)$ and $w(x),\ x\in(-1,1)$\label{fig:weight}}
\end{figure}

Then for any $i=0,\dots,n$, we have (see also (BS-5)):
$$W(|x|)=W(x)=B_{i,4}(x+i+2)=B_{i-2,4}(x+i), \quad x\in[-2,2].$$
Hence $W(|x-i|)=B_{i-2,4}(x)$, $x\in [i-2,i+2]$.

The least-squares error is (see also conditions (H1): $p_1(x)=1$, so $p(x)=p$ has to be a constant) 
$$
\sum_{i=1}^{n+r}W(x-i)\left(p-f(i)\right)^2=
\sum_{i=i_0-1}^{i_0+2}W(x-i)\left(p-f(i)\right)^2,
$$
because $W(x-i)>0$, iff $i_0-1\leq i\leq i_0+2$.

It is not hard to compute 
\begin{align*}
E=& \begin{pmatrix}
1\\
1\\
1\\
1\\
\end{pmatrix},\ \boldsymbol c = \begin{pmatrix}
1
\end{pmatrix},\\
D= & 2 {\small\begin{pmatrix}
w(|x-(i_0-1)|) & 0 & 0 & 0\\
0 & w(|x-i_0|) & 0 & 0\\
0 & 0 & w(|x-(i_0+1)|) & 0\\
0 & 0 & 0 & w(|x-(i_0+2)|)\\
\end{pmatrix}},
\end{align*}
where $w(x)=\dfrac{1}{W(x)}$, and
\begin{align*}
E^tD^{-1}E =&\dfrac{1}{2}\left(\dfrac{1}{w(|x-(i_0-1))|}+\dfrac{1}{w(|x-i_0|)}+\dfrac{1}{w(|x-(i_0+1)|)}\right.\\
&\left.\qquad +\dfrac{1}{w(|x-(i_0+2)|)}\right)\\
=&\dfrac{1}{2}\left(B_{i_0-3,4}(x)+B_{i_0-2,4}(x)+B_{i_0-1,4}(x)+B_{i_0,4}(x)\right)\\
=&\dfrac{1}{2},\quad \text{because}\ x\in[i_0,i_0+1],\\
\boldsymbol a =& D^{-1}E\left(E^tD^{-1}E\right)^{-1} \boldsymbol c = \begin{pmatrix}
B_{i_0-3,4}(x)\\
B_{i_0-2,4}(x)\\
B_{i_0-1,4}(x)\\
B_{i_0,4}(x)
\end{pmatrix}.
\end{align*}

Hence, by Theorem \ref{thm_Levin1}, we have
$$
\hat L(f)(x)=\sum\limits_{i=1}^{4}a_if(x_i)=\sum\limits_{i=i_0-3}^{i_0}B_{i,4}(x) f(i) ,
$$
i.e. we received b-spline \eqref{thm1_temp_7}.
\end{proof}

\begin{remark}\label{sec:2:rem:1} Using the method, illustrated in the proof of Theorem \ref{sec_2_thm_1}, it is not difficult to generalize the result in whole interval:
$$
\gamma_2 (x)=\sum\limits_{i=0}^{n}W(x-i)f(i),\quad x\in[r-1,n].
$$

See Example \ref{sec:2:ex:1} in Section \ref{sec:2}.
\end{remark}

\begin{remark}\label{sec:2:rem:2} Using mentioned in Subsection 1.2, Levin's approach (i.e. working with weight-function $w(x)$, such that $w(x_i)=0$) in moving least-squares method, it is not difficult to receive interpolation. Let
$$
\widetilde W(x)=W(x)+\delta,\quad x\in{\mathbb R},
$$
where $\delta>0$. Then $\widetilde W(x)\geq\delta>0$, for any $x\in{\mathbb R}$ and
$\max\{\widetilde W(x):x\in{\mathbb R}\}=\widetilde W(0)=\frac{2}{3}+\delta$. Let moreover
$$
\widetilde w(x)=\dfrac{1}{W(x)+\delta}-\dfrac{3}{2+3\delta},\quad x\in{\mathbb R}.
$$
\begin{figure}[t!]
\centerline{\includegraphics[scale=.3]{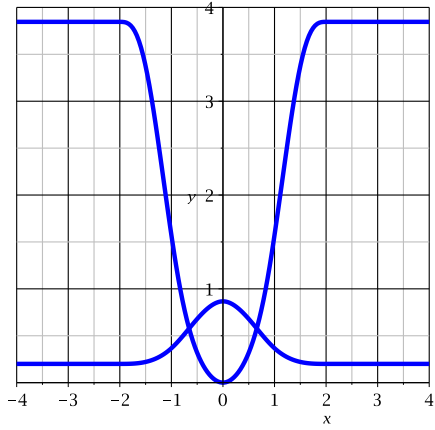}}
\caption{The graphics of $\widetilde W(x)$ and $\widetilde w(x)$, $x\in(-4,4)$\label{fig:weight:0}}
\end{figure}

Then: $\widetilde w(x)>0$, for any $x\in{\mathbb R}\setminus\{0\}$ and
$\min\{\widetilde w(x):x\in{\mathbb R}\}=\widetilde w(0)=0$, see Figure \ref{fig:weight:0}.

In this case the method used in the proof of Theorem \ref{thm_Levin1} produces interpolation.

See Example \ref{sec:2:ex:2} in Section \ref{sec:2}.
\end{remark}

\section{Some Examples. Case of Interpolation}\label{sec:2}

\begin{example}\label{sec:2:ex:1}
Consider the following example of control points:
\begin{multline*}
\Xi_0=\{(x_i,f(x_i)):x_i=i,\\ f(x)=e^{-x^2}+3e^{-(x-4)^2}+1.7e^{-(x-8)^2},\ i=0,\dots,10\}.
\end{multline*}

Here $n=10$, $r=4$, knots: $t_i=i$, $i=1,\dots,14$. The control points and cubic b-spline $\boldsymbol \gamma(x)$ are illustrated on Figure \ref{fig:data:1}. Here we used standart Maple expression {\tt BSplineCurve}, see \cite{maple}.
\begin{figure}[t!]
\centerline{\includegraphics[scale=.3]{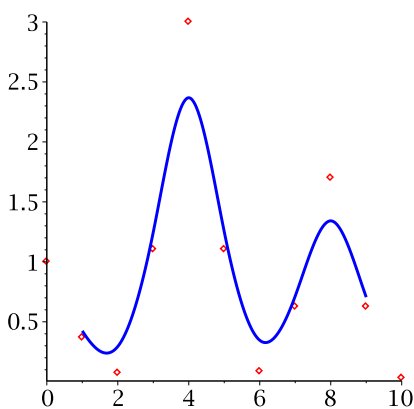}}
\caption{The plots of data-set $\Xi_0$ and cubic b-spline in $[1,9]$\label{fig:data:1}}
\end{figure}

Using Remark \ref{sec:2:rem:1} it is easy to construct the b-spline curve in whole interval $[2,10]$ -- see Figure \ref{fig:data:1:bspline}.
\begin{figure}[t!]
\centerline{\includegraphics[scale=.3]{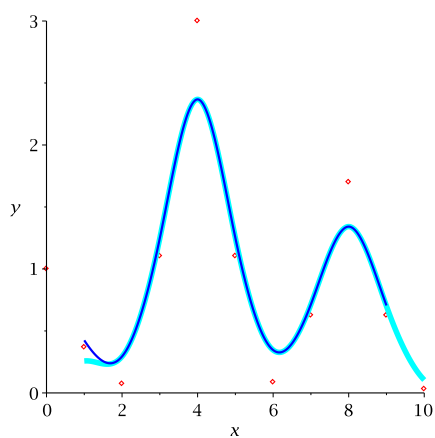}}
\caption{The plots of data-set $\Xi_0$ (red dots), cubic b-spline in $[1,9]$ (blue curve) and moving least square approximation in $[2,10]$ (cyan bold curve)\label{fig:data:1:bspline}}
\end{figure}

Let us apply Theorem \ref{sec_2_thm_1} in the interval $(5,6)$, for example. We have $W(x)=B_{5-2,4}(x+5+2)$ and moreover
$$
W(x)>0,\ i=4,5,6,7;\quad W(x)=0,\ i=0,1,2,3,8,9,10.
$$

Following the proof of theorem: let $w(s)=\dfrac{1}{W(s)}$.
 Let $x\in(5,6)$ be a fixed point, then:
\begin{align*}
E^tD^{-1}E=&\dfrac{1}{2}\left(\dfrac{1}{w(|x-4|)}+\dfrac{1}{w(|x-5|)}+\dfrac{1}{w(|x-6|)}+\dfrac{1}{w(|x-7|)}\right)\\=&\dfrac{1}{2},
\end{align*}
see (BS-3). Moreover
\begin{gather*}
\boldsymbol a = D^{-1}E\left(E^tD^{-1}E\right)^{-1} \boldsymbol c 
=  D^{-1}E\ 2
=\begin{pmatrix}
B_{2,4}(x)\\
B_{3,4}(x)\\
B_{4,4}(x)\\
B_{5,4}(x)\\
\end{pmatrix}.
\end{gather*}
So, we receive the classical b-spline formula in the interval $(5,6)$:
$$
\gamma_2(x)=\sum\limits_{i=2}^{5}W(x-i)f(i)=\sum\limits_{i=2}^{5}B_{i,4}(x)f(i).
$$

The b-splines in the intervals $[5,6]$, $[6,7]$, and $[7,8]$ are ploted on Figure \ref{fig:data:k=5,6,7}, respectively.
\begin{figure}[t!]
\centerline{
\includegraphics[scale=.2]{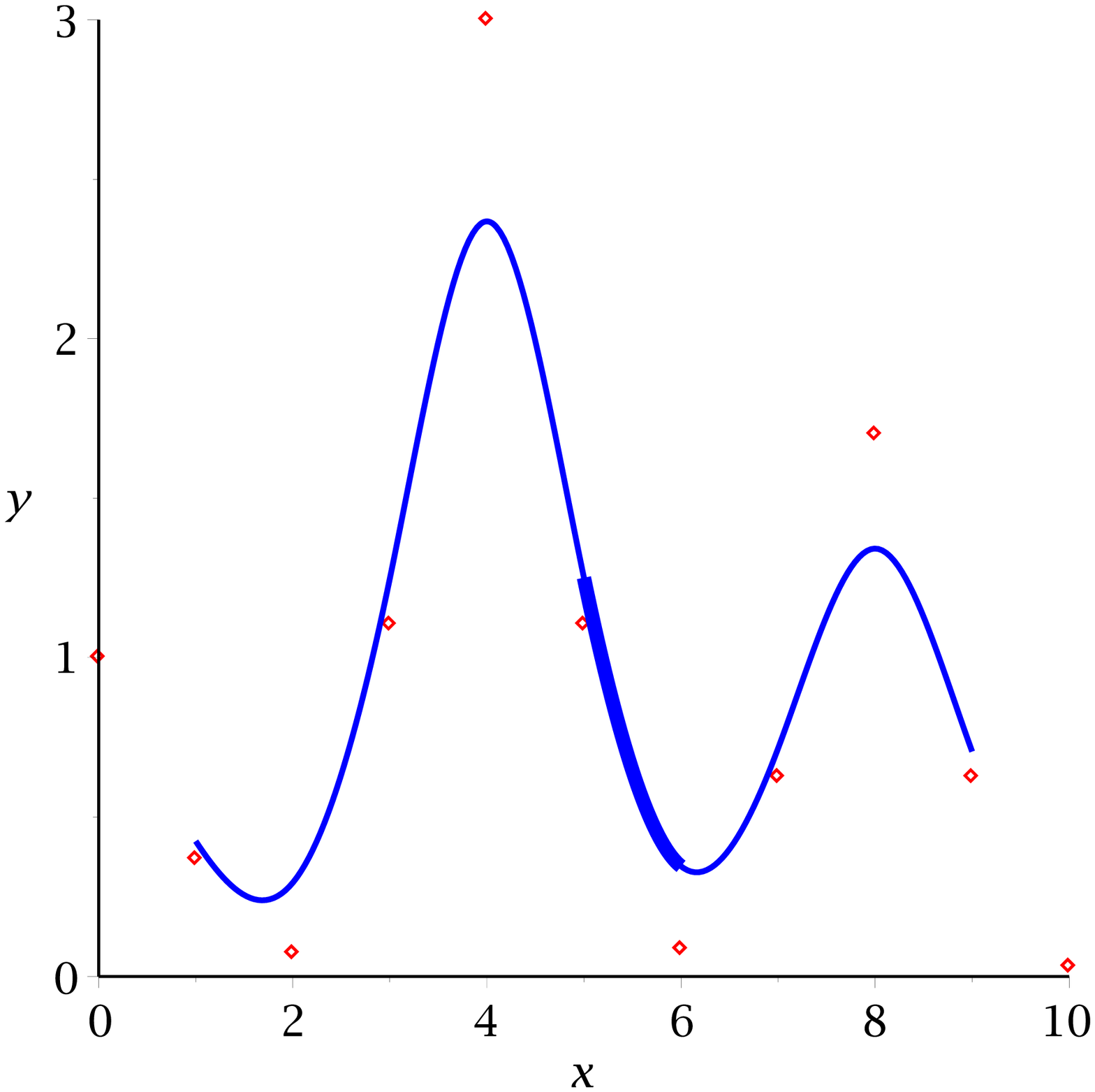}
\includegraphics[scale=.2]{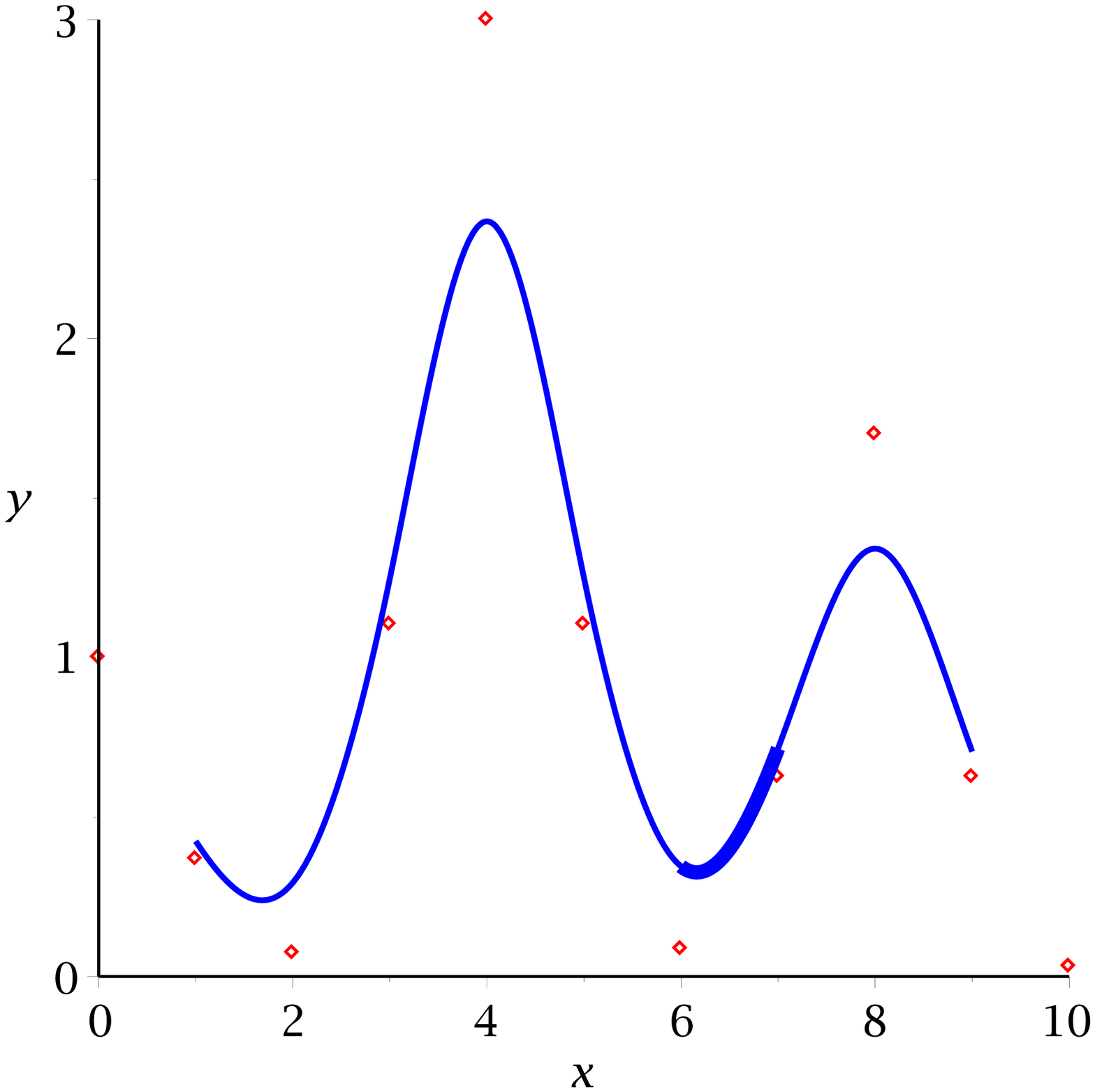}
\includegraphics[scale=.2]{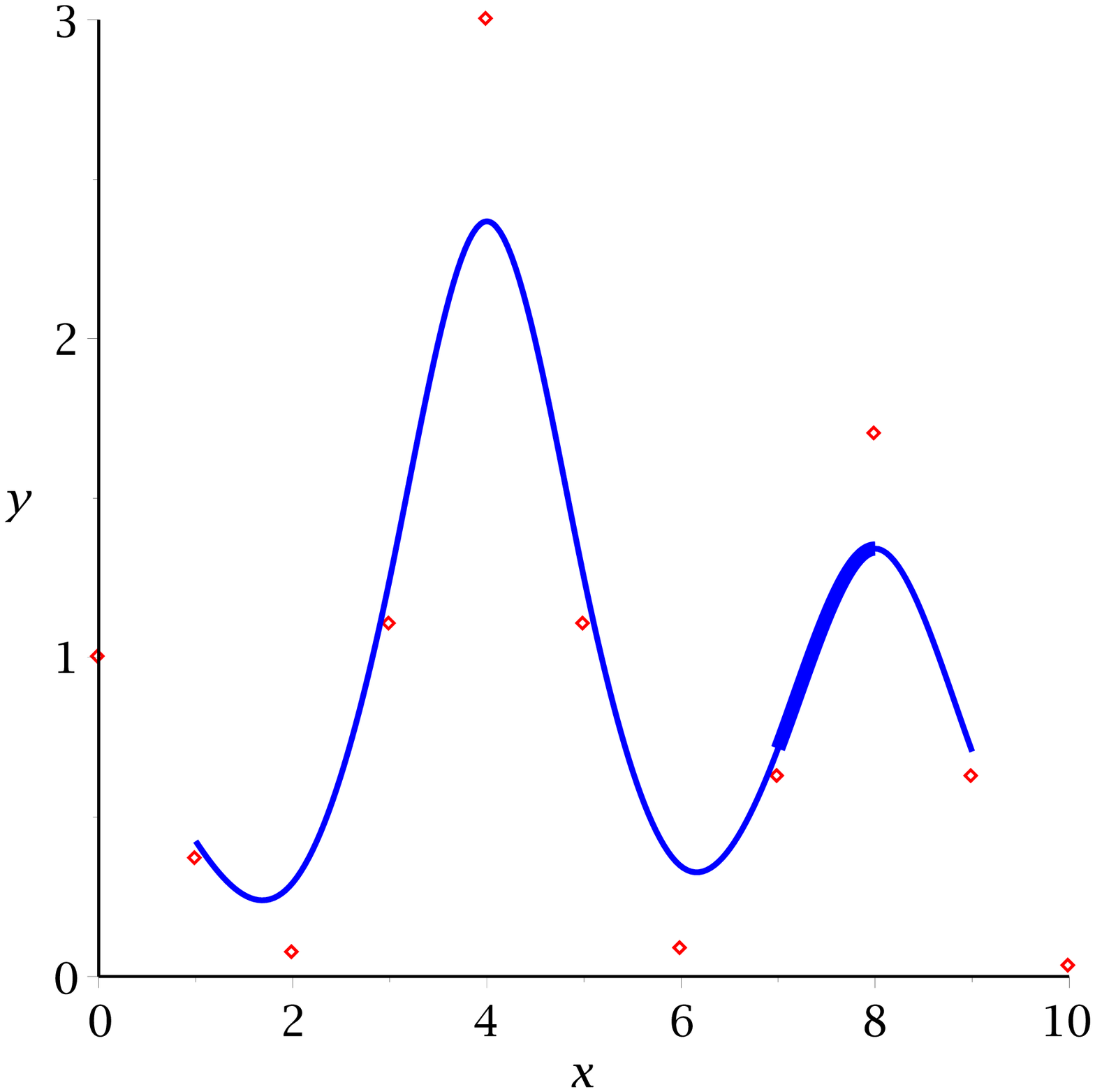}
}
\caption{The plots of b-spline in $[5,6]$, $[6,7]$, and $[7,8]$, respectively (blue bold curve)\label{fig:data:k=5,6,7}}
\end{figure}

\end{example}

\begin{example}\label{sec:2:ex:2}
To construct the interpolation in the interval $[x_0,x_{10}]$, following Remark \ref{sec:2:rem:2}, let us set $\delta=0.1$, for example. Then
$$
\widetilde w(x)=\dfrac{1}{W(x)+0.1}-\dfrac{3}{5},\quad x\in{\mathbb R}.
$$

Applying the moving least-squares method (i.e. applying Remark \ref{sec:2:rem:2} and Theorem \ref{thm_Levin1} at each point $x=l/100$, $l=1,\dots,1000$) we received the function presented in Figure \ref{interpolation:1}.
\begin{figure}[t!]
\centerline{\includegraphics[scale=.3]{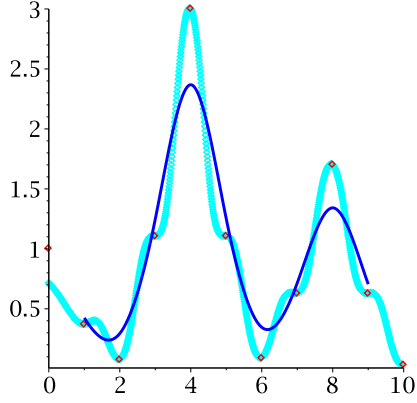}}
\caption{The plots of data (red dots), b-spline (blue), and interpolation (cyan, bold)\label{interpolation:1}}
\end{figure}

%Let us mark the changees of b-splines according parameter $\delta$ in weight function $\widetilde w(x)=\dfrac{1}{W(x)+\delta}$, i.e. in this case $\widetilde w(0)\not=0$. Some plots are presented in Figure \eqref{}.
\end{example}

\section{B-Spline Surfaces}

Let $\{\boldsymbol p_{ij}\in{\mathbb R}^{3}: i,j=0,\dots,n\}$ be a set of $(n+1)^2$ control points.

Let $r$ be an integer, $1\leq r \leq n+1$ (the order of spline).

We will use again the uniform knots, without multiplicity: $u_i=v_i=i$, $i=0,\dots,n+r$. Then, the corresponding b-spline surface of order $r$ is given by
$$
\boldsymbol r(u,v)=\sum\limits_{i=0}^{n}\sum\limits_{j=0}^{n} B_{i,r}(u)B_{j,r}(v)\boldsymbol p_{ij} .
$$

Arguments similar to the proof of Theorem \ref{sec_2_thm_1}, yield to the following result.
\begin{theorem}\label{sec_4_thm_1} Let:
\begin{enumerate}
\item $d=2$, $n,r\in{\mathbb Z_+}$, $r\leq n+1$, $f:[0,n+r]\times[0,n+r]\to\mathbb R$ be a continuous function.
\item $\boldsymbol p_{ij}=(i,j,f(i,j))$, $i=0,\dots,n+r$.
\item Let $\boldsymbol \gamma (u,v)=\left(\begin{smallmatrix}
u\\v\\\gamma_3(u,v)
\end{smallmatrix}\right)$ be the b-spline of order $r$ and knots $\{(u_i,v_j)=(i,j): i,j=0,\dots,n+r\}$.
\end{enumerate}

Then there exists a weight function $W(x,y)$, such that
$$
\gamma_3 (x,y) = \hat L(f)(x,y),\quad (x,y)\in[r-1,n]\times[r-1,n].
$$
\end{theorem}

\begin{example} As an example, consider the following data
\begin{multline*}
\Xi_0=\left\{(x_i,y_j,f(x_i)):x_i=i,\ y_j=j,\right.\\ 
\left.f(x_i,y_j)=e^{-x^2}+3e^{-(y-1)^2}+e^{-(x-6)^2-(y-6)^2},\ i=0,\dots,10\right\}.
\end{multline*}

\begin{figure}[t!]
\centerline{\includegraphics[scale=.5]{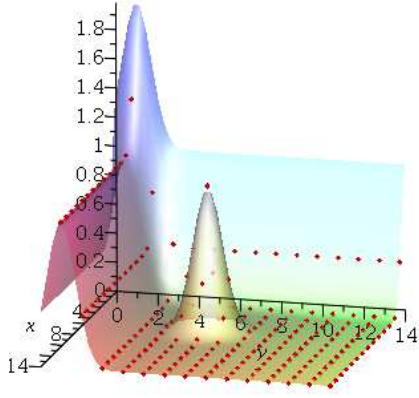}}
\caption{Set of knots and plot of function $f(x,y)$, $x,y\in[0,14]$\label{3d:1}}
\end{figure}
\begin{figure}[t!]
\centerline{\includegraphics[scale=.5]{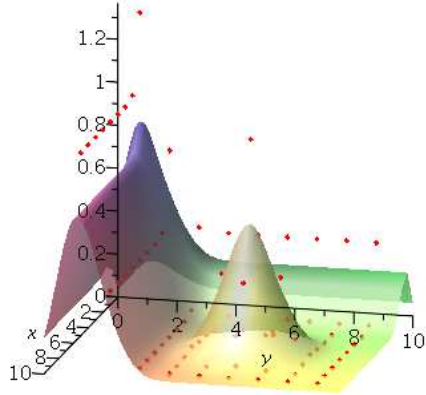}}
\caption{B-spline surface $\gamma_3(u,v)$ in $[1,9]\times[1,9]$\label{3d:2}}
\end{figure}
\begin{figure}[t!]
\centerline{
\includegraphics[scale=.3]{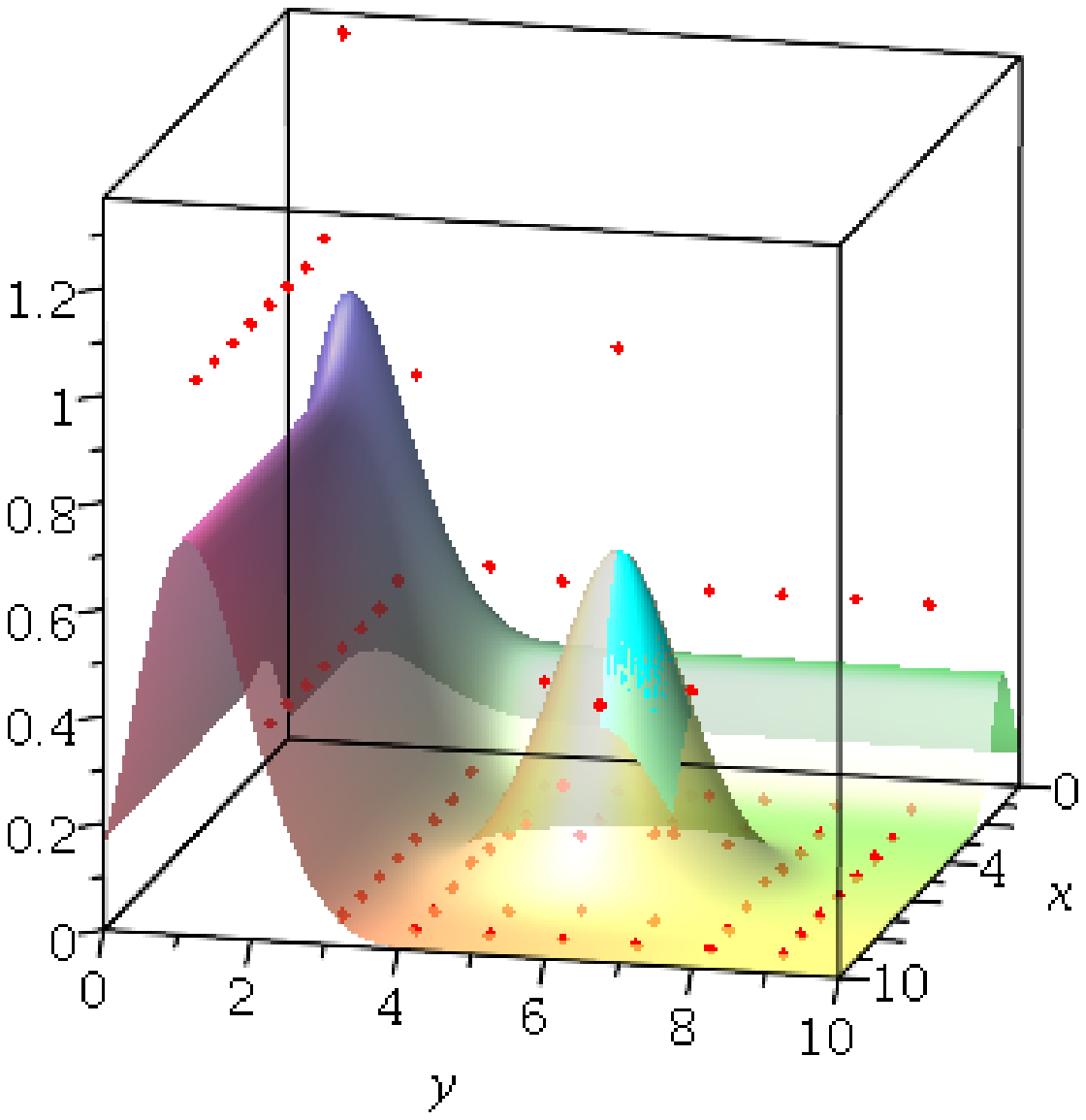}
\includegraphics[scale=.3]{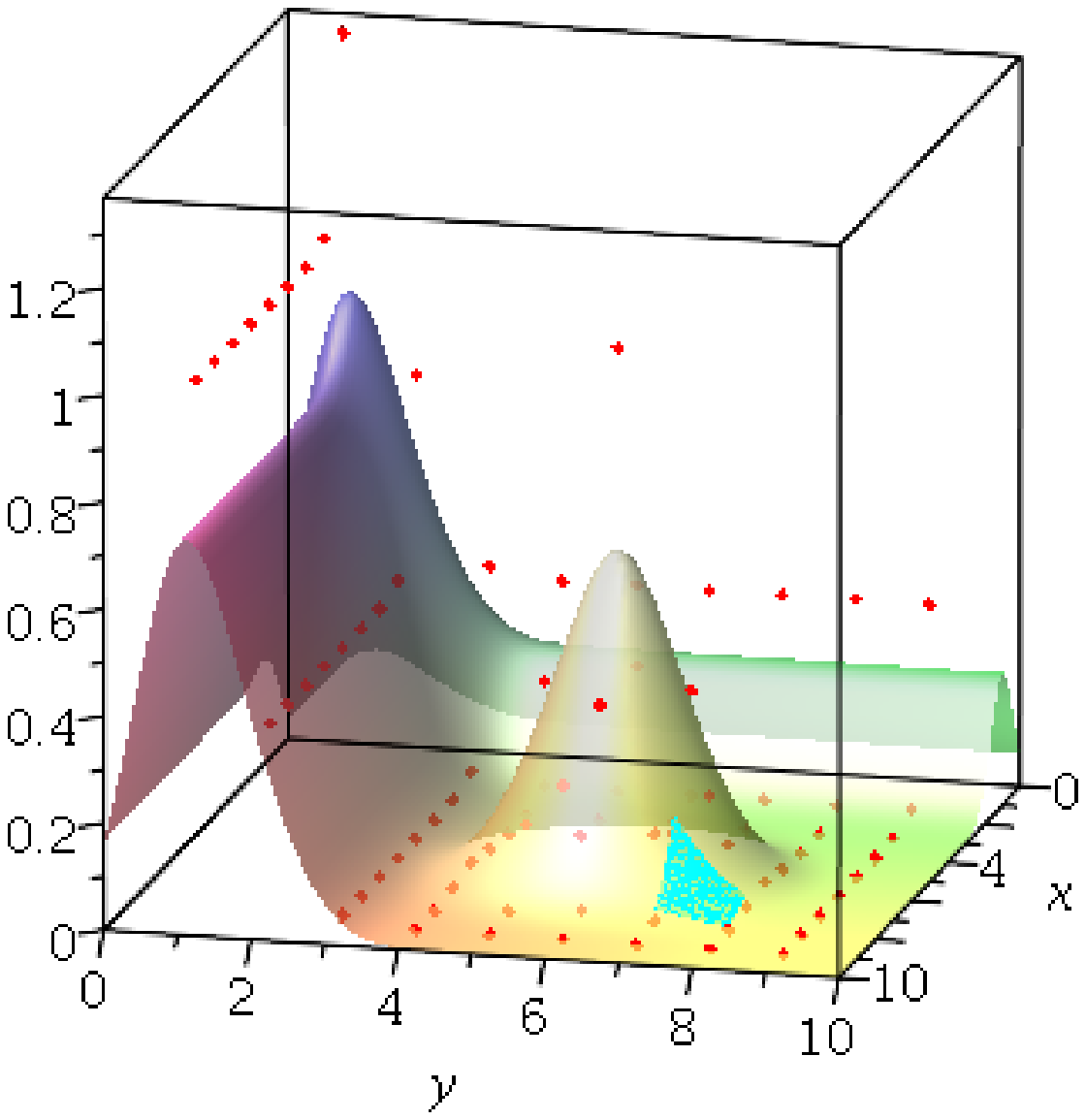}
\includegraphics[scale=.3]{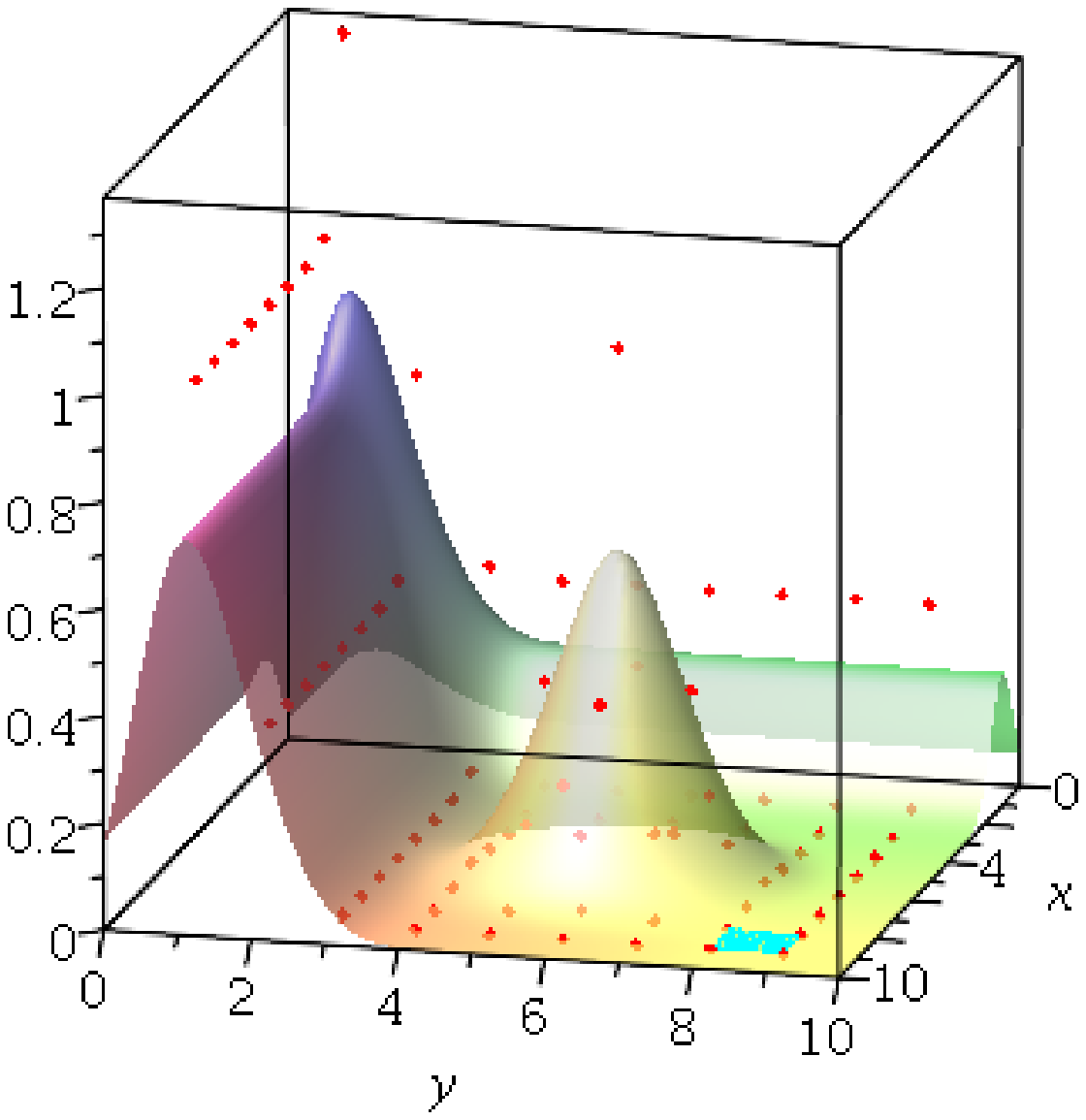}
}
\centerline{
\includegraphics[scale=.3]{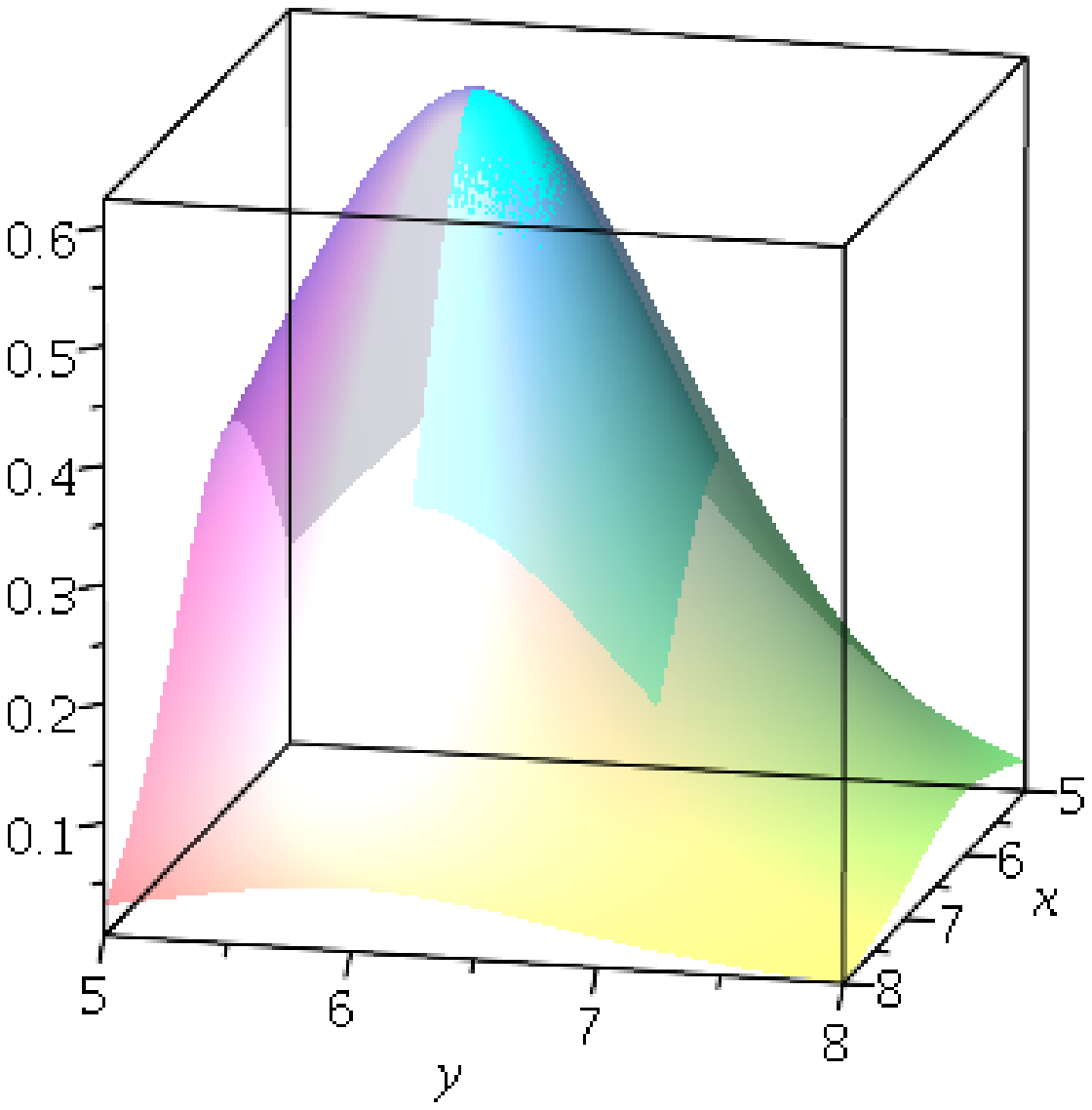}
\includegraphics[scale=.3]{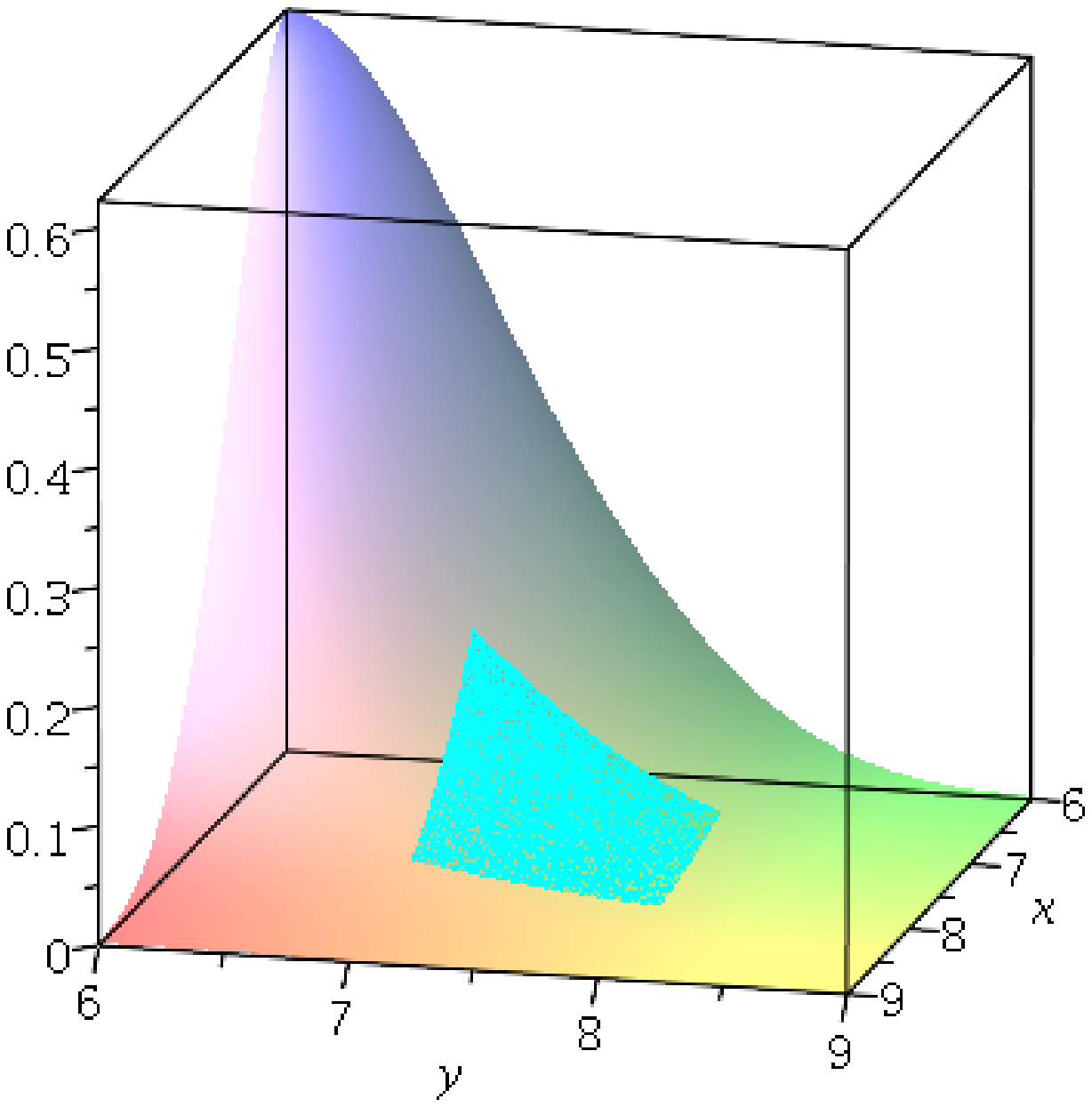}
\includegraphics[scale=.3]{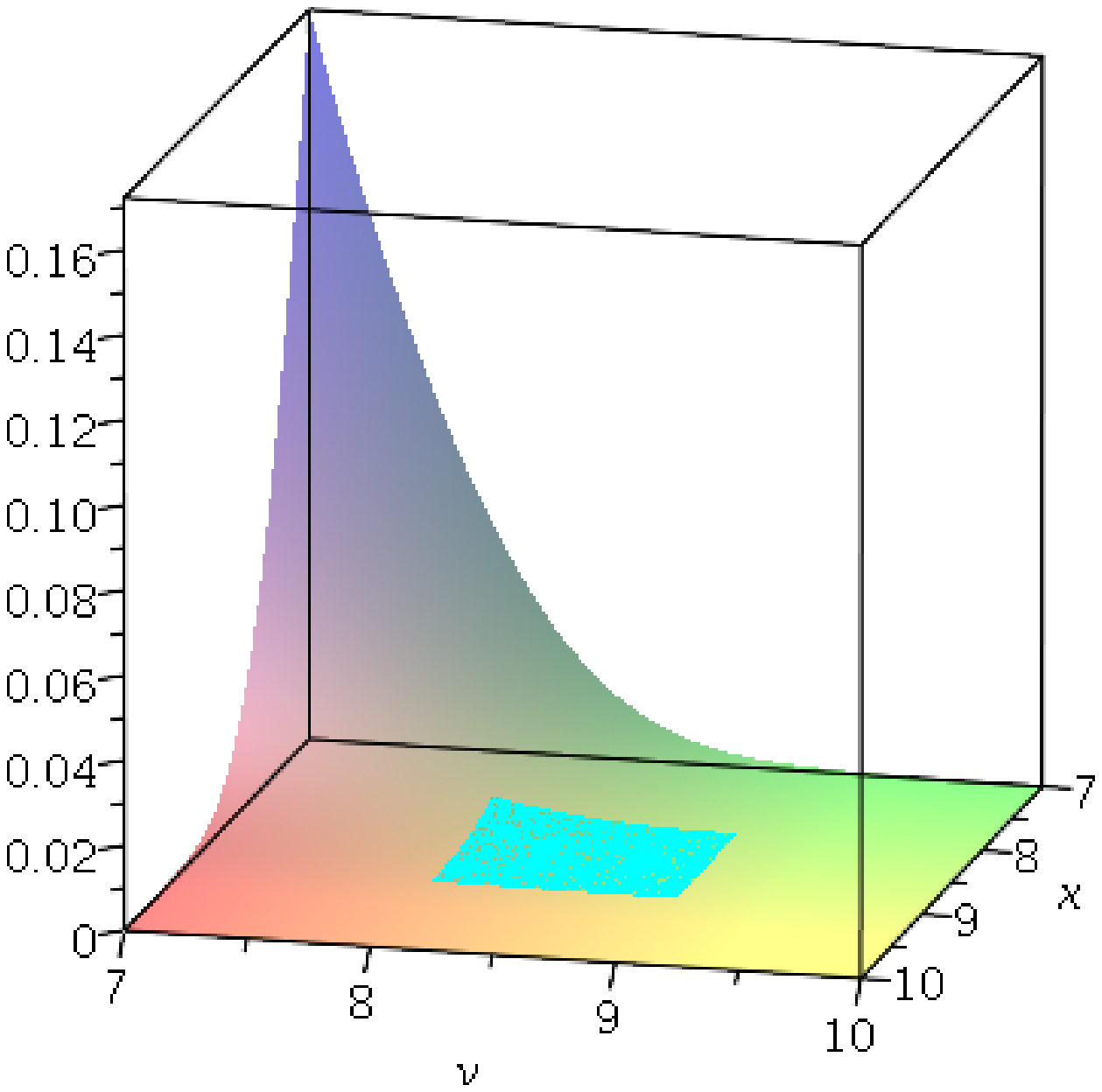}
}
\caption{B-splines in $[i,i+1]\times[i,i+1]$, $i=6,7,8$\label{3d:3} and corresponding zoomed segments}
\end{figure}
The set of control points $\boldsymbol p_{ij}=\left(\begin{smallmatrix}
x_i\\y_j\\f(x_i,y_j)\gamma_2(t)
\end{smallmatrix}\right)$ and function $f(x,y)$ are plotted in Figure \ref{3d:1}.

Using the formula
$$
\hat L(f)(x,y)=\sum\limits_{i=0}^{n}\sum\limits_{j=0}^{n}\boldsymbol p_{ij} W(x-i)W(x-j).
$$
we receive the plot of b-spline surface $\gamma_3(u,v)$ in $[1,9]\times[1,9]$, see Figure \ref{3d:2}.

If we need the plot of b-spline only in the segment $[i_0,i_0+1]\times[j_0,j_0+1]$, then
$$
\hat L(f)(x,y)=\sum\limits_{i=i_0-3}^{i_0}\sum\limits_{j=j_0-3}^{j_0}\boldsymbol p_{ij} W(x-i)W(y-j),
$$
see Figure \ref{3d:3}.
\end{example}

\end{document}